\documentclass[a4paper,11pt]{amsart}

\usepackage{amsmath,amssymb}
\usepackage{amscd}
\usepackage[all,cmtip]{xy}
\usepackage{framed}
\usepackage{hyperref}
\usepackage[inline]{enumitem}  
\usepackage{float}
\usepackage{multicol}
\usepackage{array}
\usepackage{makecell}
\newcolumntype{P}[1]{>{\centering\arraybackslash}p{#1}}
\newcolumntype{M}[1]{>{\centering\arraybackslash}m{#1}}
\usepackage{lscape}
\usepackage{times}

\usepackage{setspace}

\theoremstyle{plain}
\newtheorem{thm}{Theorem}[section]
\newtheorem{theorem}[thm]{Theorem}

\newtheorem{lemma}[thm]{Lemma}

\newtheorem{proposition}[thm]{Proposition}

\theoremstyle{definition}

\newtheorem{remark}[thm]{Remark}

\newtheorem*{claim*}{Claim}
\newtheorem*{claimproof*}{Proof of Claim}
\newtheorem*{notation*}{Notation}

\newtheorem{example}[thm]{Example}

\numberwithin{equation}{section}

\newcommand{\Q}{{\mathbb Q}}
\newcommand{\R}{{\mathbb R}}
\newcommand{\Z}{{\mathbb Z}}

\newcommand{\Aut}{\operatorname{Aut}}

\newcommand{\rank}{\operatorname{rank}}
\newcommand{\ord}{\operatorname{ord}}
\newcommand{\disc}{\operatorname{disc}}
\newcommand{\id}{\operatorname{id}}

\newcommand{\trace}{\operatorname{trace}}
\newcommand{\transpose}{\operatorname{Tr}}

\begin{document}

\title[Lucas sequences, Pell's equations, and automorphisms of K3 surfaces]{Lucas sequences, Pell's equations, and automorphisms of K3 surfaces}


\author{Kwangwoo Lee}
\thanks{This paper has been published in 
\emph{The Ramanujan Journal}, Vol. 68 (2025), \url{https://doi.org/10.1007/s11139-025-01257-6}. This work was carried out while the author was at Korea Institute for Advanced Study.}

\address{Center for Complex Geometry, Institute for Basic Science (IBS), Daejeon 34126, Republic of Korea}
\email{klee0222@gmail.com}


\maketitle

\begin{abstract}
We have the correspondences between Lucas sequences, Pell's equations, and the automorphisms of K3 surfaces with Picard number $2$. Using these correspondences, we determine the intersections of some Lucas sequences.
\end{abstract}

\section{Introduction}
In \cite[Theorem 1.1]{L2} (see Theorem \ref{Case Q=-1}) , the author determined the automorphism group of a K3 surface $X_{L_m(a)}$ whose Picard lattice $L_m(a)$ has the intersection matrix
\begin{equation}\label{L_m(a)}
m\begin{pmatrix} 2 & a \\ a & -2 \end{pmatrix},~\ a\in \Z_{\geq 1} ~\ \text{ and } ~\ m\in \Z_{\geq 2},
\end{equation}
using the properties of the generalized Fibonacci sequence $\{a_n\}$ such that 
\begin{equation}\label{a_n}
a_0=0, ~\ a_1=1, ~\ \text{ and } ~\ a_{n}=a\cdot a_{n-1}+a_{n-2}, ~\ n\geq 2.
\end{equation}
Moreover, in \cite{GLP}, Galluzzi, Lombardo, and Peters found the automorphism group of a K3 surface $X_{L(b)}$ whose Picard lattice $L(b)$ has the intersection matrix
\begin{equation}\label{L(b)}
\begin{pmatrix} 2& b\\ b& 2\end{pmatrix},~\ b\in \Z_{\geq 4}.
\end{equation}

 These results give us the correspondences (see Theorem \ref{correspondence of three items}) between the automorphisms of K3 surfaces with the Picard lattice $L_m(a)$ (or $L(b)$), the generalized Fibonacci numbers $a_n$ (or $b_n$), and the solutions of the $y$-coordinate of \textit{Pell's equation} $x^2-(a^2+4)y^2=\pm 4$ (or $x^2-(b^2-4)y^2=4$) (see Section \ref{section of Pell's equations} for Pell's equation). 
Here $\{b_n\}$ is the generalized Fibonacci sequence such that for $b\in \Z_{\geq 4}$,
\begin{equation}\label{b_k}
b_0=0, ~\ b_1=1, ~\ \text{ and } ~\ b_{n}=b\cdot b_{n-1}-b_{n-2}, ~\ n\geq 2.
\end{equation}
 \begin{theorem} \label{correspondence of three items}
There are the correspondences between the following sets.
\begin{enumerate}
\item The set of all generalized Fibonacci numbers $a_n$ ($b_n$, respectively) $n\geq 2$.
\item The set of all solutions of the $y$-coordinate of $x^2-(a^2+4)y^2=\pm 4$ ~\  $(x^2-(b^2-4)y^2=4$, respectively).
\item The set of all pairs $(X_{L_m(a)},g)$ of a K3 surface $X_{L_m(a)}$ and its automorphism $g$ for some $m (\geq 2)$ with $\trace g^*\vert_{L_m(a)}=(a^2+4)a_n^2+(-1)^n 2$ and $g^*\omega_X=(-1)^n\omega_X$ 
~\ ($(X_{L(b)},g)$ with $\trace g^*\vert_{L(b)}=(b^2-4)b_n^2+2$ and $g^*\omega_X=\omega_X$, respectively). 
\end{enumerate}
\end{theorem}
~\

Moreover, using these relations, we solve some systems of Pell's equations:
\begin{equation}\label{system 1}
  \left\{\begin{array}{@{}l@{}}
    x^2-D_1y^2=\epsilon_1 4\\
    x^2-D_2z^2=\epsilon_2 4
  \end{array}\right.\,,
\end{equation}
where $D_i$ is either $a^2+4$ or $b^2-4$, and $\epsilon_i$ is either $1$ or $-1$ ($i=1,2$) (see Section \ref{systems of Pell's equations}).

The structure of this paper is the following: in Section \ref{preliminaries}, we recall some results on Lucas sequences, lattices, K3 surfaces, and Pell's equations. In Section \ref{proof}, we prove Theorem \ref{correspondence of three items}. In Section \ref{systems of Pell's equations}, we solve some systems of Pell's equations given as in \eqref{system 1}.

\section{Preliminaries}\label{preliminaries}
\subsection{Lucas sequences}
Our generalized Fibonacci sequences $\{a_k\}$ and $\{b_k\}$ are some kinds of Lucas sequences defined as follows. The pair of \textit{Lucas sequences} $\{U_n(P, Q)\}$ and $\{V_n(P, Q)\}$ are defined by the same linear recurrent relation with the coefficients $P, Q\in \Z$ but different initial terms:
\begin{equation*}
\label{Lucas seq}
\begin{gathered} 
    U_0(P,Q)=0, ~\ U_1(P,Q)=1, ~\ U_{n+1}(P,Q)=PU_n(P,Q)-QU_{n-1}(P,Q),~\ n\geq 1;\\
     V_0(P,Q)=2, ~\ V_1(P,Q)=P, ~\ V_{n+1}(P,Q)=PV_n(P,Q)-QV_{n-1}(P,Q), ~\ n\geq 1.
  \end{gathered}
\end{equation*}
For example, $U_n(1,-1)$ are the Fibonacci numbers, $V_n(1,-1)$ are the Lucas numbers, $U_n(2,-1)$ are the Pell numbers, and $V_n(2,-1)$ are the Pell-Lucas numbers. Note that the generalized Fibonacci sequences $\{a_n\}=\{U_n(a,-1)\}$ and $\{b_n\}=\{U_n(b, 1)\}$.

The characteristic polynomial of the Lucas sequences $\{U_n(P, Q)\}$ and $\{V_n(P, Q)\}$ is $x^2-Px+Q$ with the \textit{discriminant} $D:=P^2-4Q$. A Lucas sequence $\{U_n(P, Q)\}$ or $\{V_n(P,Q)\}$ is called \textit{degenerate} if $n$-th term is $0$ for some positive $n$. Therefore, a sequence that is not zero for all $n\geq 1$ is called \textit{non-degenerate}.  For non-degenerate sequences, the discriminant $D$ is a positive non-square integer.

Let $\alpha=\frac{P+\sqrt{D}}{2}$ and $\beta=\frac{P-\sqrt{D}}{2}$ be the roots of the characteristic polynomial, then the following explicit (Binet-type) formulas take place
\begin{equation}
\begin{gathered}
  Q=\alpha\beta, ~\ D=(\alpha-\beta)^2, ~\ V_n(P,Q)=\alpha^n+\beta^n,\\
  U_n(P,Q)=\frac{\alpha^n-\beta^n}{\alpha-\beta} ~\ ~\ \text{ if } ~\ \alpha\neq \beta ~\ \text{ or } ~\ U_n(P,Q)=n\alpha^{n-1} ~\ \text{ if } ~\ \alpha=\beta.
\end{gathered}
\end{equation}
In particular, these formulas imply that
\begin{equation}
V_n(P,Q)^2-DU_n(P,Q)^2=4Q^n.
\end{equation}

For $|Q|=1$, the pairs $(V_n(P,Q),U_n(P,Q))$ form the solutions to Pell's equation:
\begin{equation}\label{Pell's equation}
x^2-Dy^2=\pm 4.
\end{equation}

The following theorem generalizes the result of Whitney; in \cite{W}, he showed that a given integer $n$ is a Fibonacci number if and only if either $5n^2+4$ or $5n^2-4$ is a square number. 
\begin{theorem}\cite[Theorem 1.3]{L2}\label{theorem1}
$n$ is the $k$-th generalized Fibonacci number $a_k$ and $k$ is even (resp. odd) if and only if $(a^2+4)n^2+4$ (resp. $(a^2+4)n^2-4$) is a perfect square.
\end{theorem}

\begin{remark}\label{Whitney theorem for b_k}
Theorem \ref{correspondence of three items} also gives the criterion that a given integer $n$ is the $k$-th generalized Fibonacci number $b_k$ if and only if $(b^2-4)n^2+4$ is a perfect square.
\end{remark}



\subsection{Lattices}
A \textit{lattice} is a free $\Z$-module $L$
 of finite rank equipped with a symmetric bilinear form
 $\langle ~,~ \rangle \colon L\times L\rightarrow \Z$.
If $x^2:=\langle x,x\rangle \in 2\Z$ for any $x\in L$,
 a lattice $L$ is said to be even. 
We fix a $\Z$-basis of $L$ and identify the lattice $L$
 with its intersection matrix $Q_L$ under this basis.
The \textit{discriminant} $\disc(L)$ of $L$ is defined as $\det(Q_L)$,
 which is independent of the choice of basis.
A lattice $L$ is called \textit{non-degenerate} if
 $\disc(L)\neq 0$ and \textit{unimodular} if $\disc(L)=\pm1$.
For a non-degenerate lattice $L$,
 the signature of $L$ is defined as $(s_+,s_-)$,
 where $s_{+}$ (resp.\ $s_-$)
 denotes the number of the positive (resp.\ negative) eigenvalues of
 $Q_L$.
An \textit{isometry} of $L$ is an automorphism of the $\Z$-module $L$
 preserving the bilinear form.
The \textit{orthogonal group} $O(L)$ of $L$ consists of the isometries of $L$ and we have
 the following identification:
\begin{equation} \label{EQ_ortho}
 O(L)= \{ g\in GL_n(\Z) \bigm| g^T \cdot Q_L \cdot g=Q_L \}, \quad
 n=\rank L.
\end{equation}
For a non-degenerate lattice $L$,
 the \textit{discriminant group} $A(L)$ of $L$ is defined by
\begin{equation} 
 A(L):=L^*/L, \quad
 L^*:=
 \{ x\in L\otimes \Q \bigm|
 \langle x,y \rangle \in \Z ~\ \forall y\in L \}.
\end{equation}
For a non-degenerate lattice $L$ of signature $(1,k)$ with $k\geq 1$,
 We have the decomposition
\begin{equation} \label{EQ_positive_cone}
 \{x\in L\otimes \R \bigm| x^2 >0 \}=C_L\sqcup (-C_L)
\end{equation}
 into two disjoint cones.
Here, $C_L$ and $-C_L$ are connected components, and $C_L$ is called the \textit{positive cone}.
We define
\begin{equation} \label{O^+}
 O^+(L):=\{g\in O(L) \bigm| g(C_L)=C_L \}, \quad
 SO^+(L):=O^+(L)\cap SO(L),
\end{equation}
 where $SO(L)$ is the subgroup of $O(L)$
 consisting of isometries of determinant $1$.
The group $O^+(L)$ is a subgroup of $O(L)$ of index $2$.

\begin{lemma}\cite[Lemma 1]{HKL}\label{criterion}
Let $L$ be a non-degenerate even lattice of rank $n$. For $g\in O(L)$ and $\epsilon\in\{\pm 1\}$, $g$ acts on $A(L)$ as $\epsilon\cdot \id$ if and only if $(g-\epsilon\cdot I_n)\cdot Q_L^{-1}$ is an integer matrix. 
\end{lemma}

\begin{proposition} (\cite[Thm 87]{D},\cite[Thm 50, Thm 51c]{J}) \label{PROP_generator_of_isometry}
Let $L$ be an even lattice of signature $(1,1)$:
\begin{equation}
 L= \begin{pmatrix} 2a& b\\ b&2c \end{pmatrix}, \quad
 D=-\disc(L)=b^2-4 ac>0.
\end{equation}
Then $SO^+(L)$ consists of the elements of the form
\begin{equation} \label{EQ_u}
 g=
 \begin{pmatrix}
  (u-bv)/2 & -cv \\
  av & (u+bv)/2
 \end{pmatrix}.
\end{equation}
Here, $(u,v)$ is any solution of the positive Pell equation
\begin{equation} \label{EQ_Pell}
 u^2-D v^2=4
\end{equation}
 with $u,kv\in\Z$ and $u>0$, where $k:=\operatorname{gcd}(a,b,c)$.
If $D$ is a square number (resp. not a square number),
 then $SO^+(L)$ is isomorphic to a trivial group (resp.\ $\Z$).
\end{proposition}

\subsection{K3 surfaces}
A \textit{K3 surface} is a simply connected compact complex manifold $X$ such that $H^0(X,\Omega_X^2)=\mathbb{C}\omega_X$, where $\omega_X$ is an everywhere non-degenerate holomorphic $2$-form on $X$. 

In \cite{HKL}, we have the following result.
\begin{proposition} \label{PROP_key_prop}
Let $X$ be a projective K3 surface with Picard lattice $NS(X)$ of Picard number $2$.
For $\varepsilon\in \{ \pm 1 \}$ and $g \in O(NS(X))$,
 the following conditions are equivalent:
\begin{enumerate}
\item
$g=\varphi^*|_{NS(X)}$ for some $\varphi \in \Aut(X)$
 such that $\ord(\varphi)=\infty$ and
 $\varphi^* \omega_X=\varepsilon \omega_X$.
\item
$NS(X)$ has no $(0)$- or $(-2)$-elements and
 $g$ is given by
\begin{equation}
 g=g_{u,v}:=\begin{pmatrix}
  (u-bv)/2 & -cv \\
  av & (u+bv)/2
 \end{pmatrix},
\end{equation}
 where
\begin{equation}\label{trace}
 NS(X)=\begin{pmatrix}2a&b\\b&2c\end{pmatrix}, \quad
 (u,v)=(\alpha^2-2\varepsilon,\alpha\beta),
\end{equation}
 and $\alpha,\beta$ are nonzero integers satisfying
\begin{equation}
 \alpha^2-D \beta^2=4\varepsilon, \quad
 D:=-\disc(NS(X))=b^2-4ac>0.
\end{equation}
\end{enumerate}
\end{proposition}

In \cite[Theorem 1]{L1}, the author showed that (1) the isometry group $O(L_m(a))\cong \Z_2\ast \Z_2$, a free product with generators 
\begin{equation}\label{A and B}
A=\begin{pmatrix} 1& 0\\ a&-1\end{pmatrix} \hspace{2.6mm} \text{ and } \hspace{2.6mm} B=\begin{pmatrix} 1& a\\ 0&-1\end{pmatrix},
\end{equation}
 and (2) $\Aut(X_{L_m(a)})\cong \Z$ for $m\geq 2$, where $L_m(a)$ is in \eqref{L_m(a)}. The generator acts on Picard lattice $NS(X_{L_m(a)})$ as $(AB)^n$ for some $n$. Moreover, the author also determined $n$ for a given $m$ as follows. 
\begin{theorem}\cite[Theorem 1.1]{L2} \label{Case Q=-1} For the K3 surface $X_{L_m(a)}$, 
\begin{enumerate}
\item if $n$ is an even number such that $m\mid a_n$, then $X$ has a symplectic automorphism $g$ such that $g^*\vert_{NS(X_{L_m(a)})}=(AB)^n$;
\item if $n$ is an odd number such that $m\mid a_n$, then $X$ has an anti-symplectic automorphism $g$ such that $g^*\vert_{NS(X_{L_m(a)})}=(AB)^n$.  
\end{enumerate}
\end{theorem}

\begin{remark}
Moreover, by \cite[Lemma 3.1]{L2}, such an $n$ exists. 
\end{remark}


\subsection{Pell's equations} \label{section of Pell's equations}
Let $d$ be any positive integer. Pell's equations associated to $d$ are
\begin{align}
u^2-dv^2=&4 \hspace{15mm}  \text{ the positive Pell's equation};\label{positive Pell equation}\\
u^2-dv^2=&-4   \hspace{10mm}\text{ the negative Pell's equation}.\label{negative Pell equation}
\end{align}
The solutions are described in the following lemma.
\begin{lemma}\label{Pell equation} (\cite{B}, Theorem 3.18)
If $d$ is a square, the only solutions of the positive Pell's equation \eqref{positive Pell equation} are $u=\pm 2$ and $v=0$.

If $d$ is not a square, let $(U,V)$ be the smallest positive solution of \eqref{positive Pell equation} and write $\alpha=\frac{1}{2}(U+V\sqrt{d}).$ Then all solutions are generated by powers of $\alpha$ in the sense that writing $\alpha^n=\frac{1}{2}(u+v\sqrt{d})$, $(u,v)$ is a new solution and all solutions can be obtained that way.

If $(U\rq{},V\rq{})$ is a minimal positive solution of \eqref{negative Pell equation} and $\beta=\frac{1}{2}(U\rq{}+V\rq{}\sqrt{d})$, then $\beta^2=\alpha$ gives the minimal solution for \eqref{positive Pell equation} and the even powers of $\beta$ thus provide all the solutions of \eqref{positive Pell equation}, while the odd powers yield all solutions of \eqref{negative Pell equation}. 
\end{lemma}

\subsection{Generalized Fibonacci sequences}
The Fibonacci sequence is $\{f_n\}_{n\geq 0}$ satisfying $f_{n+2}=f_{n+1}+f_{n}$ with the initial terms $f_0=0$ and $f_1=1$. In \cite{S}, Silvester showed many properties of the sequence by matrix representation.  For  $M_{(1,1)}=\begin{pmatrix} 0& 1\\ 1&1\end{pmatrix}$,
\begin{equation}\label{Fibonacci by M}
M_{(1,1)}^n\begin{pmatrix} 0\\1\end{pmatrix}=\begin{pmatrix} f_n\\ f_{n+1}\end{pmatrix}.
\end{equation}
By diagonalizing $M$, he also showed that
\begin{equation}
f_n=\frac{r_1^n}{P\rq{}(r_1)}+\frac{r_2^n}{P\rq{}(r_2)},
\end{equation}
where $r_i$ is the eigenvalue of the companion matrix $M$ for the polynomial $P(t)=t^2-t-1$. ($P\rq{}$ is the derivative of $P$.)
This is a special case of a general Fibonacci sequence defined as follows. Given values for the first $2$ terms, $a_0, a_1$, 
\begin{equation}\label{sequence}
a_{n+2}=c_0a_n+c_1a_{n+1},
\end{equation}
where $c_0, c_1$ are constants. Then we get a generalization of \eqref{Fibonacci by M}: for 
\begin{equation}\label{general M}
M_{(c_0,c_1)}=
   \begin{pmatrix}
   0      & 1    \\
   c_0    & c_1    
  \end{pmatrix}, ~\
M_{(c_0,c_1)}^n
   \begin{pmatrix}
   a_0  \\a_1
  \end{pmatrix}= \begin{pmatrix}
   a_n  \\a_{n+1} 
  \end{pmatrix}.
\end{equation}
For the initial terms $a_0=0, a_1=1$ with $c_0=c_1=1$, we get the original Fibonacci sequence $\{f_n\}$. This is the same as in the Lucas sequence $\{U_n(1,-1)\}$. For the initial terms $a_0=0, a_1=1$ with $c_0=1, c_1=a$, we get the Lucas sequence $\{U_n(a,-1)\}$. Similarly, for the initial terms $b_0=0, b_1=1$ with $c_0=-1, c_1=b$, we get the Lucas sequence $\{U_n(b,1)\}$.

For the sequence $\{U_n(a,-1)\}$, we have the following properties:
\begin{lemma}\cite[Lemma 9]{L1}\cite[Lamma 2.1]{L2}\label{properties of Fibonacci seq}
For the sequence in \eqref{a_n}, 
\begin{enumerate}
\item $a_{n+k}=a_ka_{n+1}+a_{k-1}a_n=a_{k+1}a_n+a_ka_{n-1}$;
\item $a_{n+1}a_{n-1}-a_n^2=\left\{ \begin{array}{ll}
      1 & \mbox{if $n$ is even},\\
      -1 & \mbox{if $n$ is odd};\end{array} \right.$ 
\item $a_{2n+1}+a_{2n-1}=(a^2+4)a_n^2+(-1)^n2$.
\end{enumerate}
\end{lemma}
Similarly, for the sequence $\{U_n(b,1)\}$, we have similar properties:
\begin{lemma} \label{properties of b_n} 
For the sequence $\{b_n\}=U_n(b,1)$, 
\begin{enumerate}
\item $b_{n+k}=b_kb_{n+1}-b_{k-1}b_n=b_{k+1}b_n-b_kb_{n-1}$;
\item $b_{n-1}b_{n+1}-b_n^2=1$;
\item $b_{2n+1}-b_{2n-1}=(b^2-4)b_n^2+2$.
\end{enumerate}
\end{lemma}
\begin{proof} The sequence $\{U_n(b,1)\}$ has a matrix representation: 
for 
\begin{equation}\label{N}
N:=N_{(-1,b)}=\begin{pmatrix} 0& 1\\ -1&b\end{pmatrix},
~\ N^n\begin{pmatrix} b_0\\b_1\end{pmatrix}=\begin{pmatrix} b_n\\ b_{n+1} \end{pmatrix}.
\end{equation}
Then $b_{n+k}=\begin{pmatrix} 0& 1\end{pmatrix}N^{n+k-1}\begin{pmatrix} 0& 1\end{pmatrix}^{\transpose}=\begin{pmatrix} 0& 1\end{pmatrix}N^nN^{k-1}\begin{pmatrix} 0& 1\end{pmatrix}^{\transpose}=b_kb_{n+1}-b_{k-1}b_n$. Similarly, $b_{n+k}=\begin{pmatrix} 0& 1\end{pmatrix}N^{n+k-1}\begin{pmatrix} 0& 1\end{pmatrix}^{\transpose}=\begin{pmatrix} 0& 1\end{pmatrix}N^{k}N^{n-1}\begin{pmatrix} 0& 1\end{pmatrix}^{\transpose}=b_{k+1}b_{n}-b_{k}b_{n-1}$.
The second follows from the determinant of $N$=1. For the last, $b_{2n+1}-b_{2n-1}=b_{n+1}^2-b_n^2-(b_n^2-b_{n-1}^2)$ by condition 1, and which is $-2b_n^2+(bb_n-b_{n-1})^2+b_{n-1}^2=(b^2-2)b_n^2-2b_{n-1}b_{n+1}=(b^2-4)b_n^2+2$ by the condition 2.
\end{proof}

\begin{remark}\label{M_a and N_b}
For $M_a:=\begin{pmatrix} 0&1\\1  & a\end{pmatrix}$ and $N_b:=\begin{pmatrix} 0&1\\-1  & b\end{pmatrix}$,
\begin{equation}\label{power of M and N}
M_a^{n}=\begin{pmatrix} a_{n-1}&a_n\\a_n  & a_{n+1}\end{pmatrix} ~\ \text{ and }~\ N_b^{n}=\begin{pmatrix} -b_{n-1}&b_n\\ -b_n  & b_{n+1}\end{pmatrix} (n\geq 1).
\end{equation}

\end{remark}

\section{Proof of Theorem \ref{correspondence of three items}}\label{proof}
\subsection{Case $a_n$.} Let $a_n$ ($n\geq 2$) be the $n$-th generalized Fibonacci number, i.e., $a_n=U_n(a,-1)$ for some $a$. Then we consider the lattice $L_m(a)$ as in \eqref{L_m(a)}, where $m\mid a_n$ and $m\in \Z_{\geq 2}$. Then by Theorem \ref{Case Q=-1}, we have a pair of K3 surface and its automorphism $(X_{L_m(a)}, g)$ with $m\mid a_n$ and $g^*\vert_{NS(X_{L_m(a)})}=(AB)^n$. Since $AB=M_a^2$ with $M_a$ in Remark \ref{M_a and N_b}, this pair satisfies condition 3 in Theorem \ref{correspondence of three items} by \eqref{power of M and N} and Lemma \ref{properties of Fibonacci seq}.

Next, for the pair $(X_{L_m(a)}, g)$ in condition 3,  $g^*\vert_{NS(X_{L_m(a)})}=(AB)^n$ for some $n$ by Theorem \ref{Case Q=-1} and the fact that $\Aut(X_{L_m(a)})\cong \Z$ for $m\geq 2$ and the generator acts on Picard lattice $L_m(a)$ as $(AB)^n$ for some $n$ (see Theorem 1 in \cite{L1}). By Proposition \ref{PROP_key_prop}, we have $\trace(AB)^n=\alpha^2-2\epsilon$ for some $\alpha$. Again by Theorem \ref{Case Q=-1}, $\epsilon=(-1)^n$. Note that $(AB)^n=M_a^{2n}$, where $M_a$ is in Remark \ref{M_a and N_b}. Hence $\trace(AB)^n=a_{2n-1}+a_{2n+1}$ and this is $(a^2+4)a_n^2+(-1)^n2$ according to Lemma \ref{properties of Fibonacci seq}. Hence, we have a solution of the $y$-coordinate of $x^2-(a^2+4)y^2=\pm 4$.

Finally, for a solution of the $y$-coordinate of $x^2-(a^2+4)y^2=\pm 4$, it is a generalized Fibonacci number by Theorem \ref{theorem1}. \qed

\subsection{Case $b_n$.}
Before we prove the theorem for $b_n$, we consider the following. As in \eqref{general M}, $U_n(b,1)$ can be represented by matrices:
\begin{equation}\label{generalized Fibonacci by matrix with Q=1}
N_{b}^n=\begin{pmatrix} -b_{n-1}& b_n\\ -b_n& b_{n+1}\end{pmatrix},
\end{equation}
where $N_b$ is in Remark \ref{M_a and N_b}.

In \cite{GLP}, Galluzzi, Lombardo, and Peters gave the following example. 
\begin{example}\cite[Example 4 in \S5]{GLP}\label{example of GLP}
Let $X_{L(b)}$ be a K3 surface whose Picard lattice $L(b)$ is in \eqref{L(b)}. If $b\geq 4$, then $\Aut(X_{L(b)})\cong \Z\rtimes \Z_2$ with generators $u$ and $v$ whose actions on $L(b)$ are given by 
\begin{equation}\label{C^2}
C^2=\begin{pmatrix} 0& -1\\ 1&b\end{pmatrix}^2=(N_b^{\transpose})^2 ~\ \text{ and } ~\  D=\begin{pmatrix} 1& b\\ 0&-1\end{pmatrix}. 
\end{equation}
Moreover, $u$ is symplectic while $v$ is anti-symplectic.
Note that if $b=3$, then $L(3)$ contains a $(-2)$ element, so $X_{L(3)}$ has a finite automorphism group by the following remark.
\end{example}
\begin{remark} It is known that for a K3 surface $X$ of Picard number $2$, $\Aut(X)$ is finite if and only if there exists a nontrivial element in the Picard lattice whose self-intersection number is $0$ or $-2$ (see \cite[Corollary 1]{GLP}).
\end{remark}
\vspace{3mm}
Now, we prove the theorem for $b_n$, i.e., $U_n(b,1)$. Let $b_n$ be the $n$-th generalized Fibonacci number, i.e., $b_n=U_n(b,1)$ for some $b$. By Lemma \ref{properties of b_n}, $(b^2-4)b_n^2+2=b_{2n+1}-b_{2n-1}$ and the latter is $\trace (C)^{2n}$. By Lemma \ref{criterion}, this is an isometry of $L(b)$ in \eqref{L(b)} acting as $\id$ on the discriminant group $A(L(b)):=L(b)^*/L(b)$. Hence, the isometry $C^{2n}$ of $L(b)$ extends to an isometry of $H^2(X_{L(b)},\Z)$ by defining $\id$ on $T_X$, the transcendental lattice of $X$ by \cite[Corollary 1.5.2]{N}. This extension extends to a symplectic automorphism of $X_{L(b)}$ by the Torelli theorem. In this case, the ample cone of $X_{L(b)}$ is just the positive cone because the Picard lattice $L(b)$ has no $(-2)$ elements. Hence, $C^{2n}$ is the corresponding isometry of a symplectic automorphism $g$ of $X_{L(b)}$ on $L(b)$, i.e., $g^*\vert_{L(b)}=C^{2n}$.  Hence, $b_n$ corresponds to the pair $(X_{L(b)}, g)$ with condition 3 in Theorem \ref{correspondence of three items}. 

Next, for a pair $(X_{L(b)}, g)$ with condition 3, 
$(b^2-4)b_n^2+2=\alpha^2-2$ for some $\alpha\in \Z_{\geq 4}$ by \cite[Main Theorem]{HKL}. Hence, the pair $(X_{L(b)}, g)$ with condition 3 corresponds to a solution of the $y$-coordinate of $x^2-(b^2-4)y^2=4$. 

Next, for a $y$-coordinate solution $n$ of $x^2-(b^2-4)y^2=4$, $(b^2-4)n^2+4=\alpha^2$ with $\alpha>3$. Then by \cite[Main Theorem]{HKL}, $\alpha^2-2=(b^2-4)n^2+2$ is a trace of some automorphism $g$ of a projective K3 surface $X$ with Picard number $2$ such that $\ord(g)=\infty$ and $g^*\omega_X= \omega_X$. Then, by Proposition \ref{PROP_key_prop}, if we let $\beta=n$, then we have a K3 surface $X_{L(b)}$ with a symplectic automorphism $g$ such that $g^*\vert_{L(b)}$ is given as in Proposition \ref{PROP_key_prop}. By Proposition \ref{PROP_generator_of_isometry}, any such form is given by any solution of $u^2-(b^2-4)v^2=4$. In particular, $(u,v)=(b,1)$ gives $C^2$ in Example \ref{example of GLP}. Hence, it is generated by $C^2$. So $g^*\vert_{L(b)}=C^{2l}$ for some $l$, and hence $\alpha^2-2=\trace C^{2l}=b_{2l+1}-b_{2l-1}=(b^2-4)b_l^2+2$. So $n=b_l$ for some $l$. \qed

\section{Intersection of Lucas sequences}\label{systems of Pell's equations}

For a system of Pell's equations \eqref{system 1},
the existence of finite solutions is known as follows.
\begin{theorem}\cite[Theorem 6]{A}\label{system of Diophantine equations}
A system of Diophantine equations
\begin{equation}
  \left\{\begin{array}{@{}l@{}}
    a_1x^2+b_1y^2+c_1z^2=d_1\\
    a_2x^2+b_2y^2+c_2z^2=d_2,
  \end{array}\right.\,
\end{equation}
where $a_i, b_i, c_i, d_i (i=1, 2)$ are integers and in the matrix of coefficients
\begin{equation}
\begin{pmatrix} a_1& b_1&c_1&d_1\\ a_2&b_2&c_2&d_2\end{pmatrix} 
\end{equation}
every $2\times 2$ minor is non-zero, has a finite number of solutions.
\end{theorem}
Using Theorem \ref{correspondence of three items}, we solve several systems of Pell's equations.

\subsection{Intersection of two $V$-sequences}
Consider a system of Pell's equations
\begin{equation}\label{first type of system}
  \left\{\begin{array}{@{}l@{}}
    x^2-D_1y^2=\pm 4\\
    x^2-D_2z^2=\pm 4
  \end{array}\right.\,,
\end{equation}
where $D_i=P_i^2-4Q_i$ ($i=1,2$), the discriminant of  $\{V_n(P_i,Q_i)\}$. The set of solutions of $x$-coordinate of \eqref{first type of system} is the intersection of $\{V_n(P_1,Q_1)\}$ and $\{V_n(P_2,Q_2)\}$.
\begin{theorem}\label{intersection of V_n}
For the system of Diophantine equations with the same signs of $4$'s
\begin{equation}\label{the first system}
  \left\{\begin{array}{@{}l@{}}
    x^2-(P_1^2+4)y^2=\pm 4\\
    x^2-(P_2^2+4)z^2=\pm 4
  \end{array}\right.\,, (P_1\neq P_2),
\end{equation}
\begin{enumerate}
\item if $(P_1^2+4)(P_2^2+4)$ is not square, then the only solution is $(x,y,z)=(2,0,0)$;
\item if $(P_1^2+4)(P_2^2+4)$ is a square,  then for any solution $(x,y,z)$, there are both even or odd $m,n\in \Z_{>0}$ such that $x^2+(-1)^m 2=\trace(M_1^{2m})=\trace(M_2^{2n})$, where 
\begin{equation*}
M_1=\begin{pmatrix} 0& 1\\ 1&P_1\end{pmatrix} \text{ and } \hspace{0.2mm} M_2=\begin{pmatrix} 0& 1\\ 1&P_2\end{pmatrix}.
\end{equation*} 

Moreover, if $m$ and $n$ are the minimal such numbers, then the solution set of $x$-coordinate forms the Lucas sequence $\{V_n(\lambda^m+(-\lambda)^{-m}, (-1)^{m})\}$, where $\lambda$ and $-1/\lambda$ are the eigenvalues of $M_1$.
\end{enumerate}
\end{theorem}

\begin{proof}
The system of equations \eqref{the first system} reduces to the following system of equations
\begin{equation}\label{reduced system}
  \left\{\begin{array}{@{}l@{}}
    x^2-(P_1^2+4)y^2=\pm 4\\
    (P_1^2+4)y^2-(P_2^2+4)z^2=0
  \end{array}\right.\,.
\end{equation}
If $(y,z)$ is non-trivial, then $(P_1^2+4)(P_2^2+4)$ is a square. Thus, if it is not square, we have only one solution $(x,y,z)=(2,0,0)$. 

For the case where $(P_1^2+4)(P_2^2+4)$ is a square, we have $x^2\mp 2=(P_1^2+4)y^2\pm 2=(P_2^2+4)z^2\pm 2$. Now, by Theorem \ref{correspondence of three items}, $x^2\mp 2$ is the trace of $M_1^{2m}=M_2^{2n}$ for some $m$ and $n$.
Moreover, the signs are determined by the parity of $m$ and $n$. Indeed, they have the same parity, and if they are even (odd, resp.), then we have $x^2-2 (x^2+2, \text{ resp.})$, and the corresponding automorphisms are both symplectic (anti-symplectic, resp.). 

Since the automorphism group $\Aut(X_L)\cong \Z$ for such a K3 surface $X_L$, where $L$ is in \eqref{L_m(a)}, we have the minimal numbers $m$ and $n$ such that the traces of $M_1^{2m}$ and $M_2^{2n}$ are the same. Note that if the trace of $M_1^{2k}=M_2^{2l}$ for $k\geq m$ and $l\geq n$, then $k$ and $l$ are the multiples of $m$ and $n$, respectively. Indeed, since $M_1^{2m}$ and $M_2^{2n}$ are similar, the eigenvalues are the same. If we let $\lambda, -\lambda^{-1}$ and $\mu, -\mu^{-1}$ be the eigenvalues of $M_1$ and $M_2$, respectively, then we have $\lambda^{2m}=\mu^{2n}$ and $\lambda^{-2m}=\mu^{-2n}$. Similarly, $\lambda^{2k}=\mu^{2l}$ and $\lambda^{-2k}=\mu^{-2l}$. Hence, the trace of $\trace M_1^{2k-2m}=\lambda^{2k-2m}+\lambda^{-2k+2m}=\mu^{2l-2n}+\mu^{-2l+2n}=\trace M_2^{2l-2n}$. Repeating the same argument, we will have $\trace M_1^{2k-4m}=\trace M_2^{2l-4n}$ and so on. So, we know that $k$ and $l$ are the multiples of $m$ and $n$ because $m$ and $n$ are the minimal numbers.

The first common solution is $(x,y,z)=(2,0,0)$, i.e., $x_0=2$. The next two solutions $x_1, x_2$ in the $x$-coordinate satisfy 
\begin{align*}
&x_1^2-(-1)^m2=\trace(M_1^{2m})=a_{2m-1}+a_{2m+1}=(a_1^2+4)a_m^2+(-1)^m2=\lambda^{2m}+\lambda^{-2m},\\
&x_2^2-(-1)^{2m}2=\trace(M_1^{4m})=a_{4m-1}+a_{4m+1}=(a_1^2+4)a_{2m}^2+(-1)^{2m}2=\lambda^{4m}+\lambda^{-4m},
\end{align*}
where $a_k=U_k(P_1, -1)$, i.e., $a_0=0, a_1=1, a_k=P_1\cdot a_{k-1}+a_{k-2}$.

Hence 
\begin{align*}
x_0=2,& ~\ x_1=\lambda^{m}+(-1)^m\lambda^{-m}, \\
x_2=\lambda^{2m}+\lambda^{-2m},& ~\ x_k=\lambda^{km}+(-1)^{km}\lambda^{-km} ~\ \text{ for } k\in \Z_{\geq 0}.
\end{align*}

So, the set of solutions in the $x$-coordinate of \eqref{the first system} is the Lucas sequence $\{V_n(\lambda^m+(-\lambda)^{-m}, (-1)^{m})\}$. Moreover, it satisfies the Pell's equation $x^2-Dy^2=\pm 4$, where $D=(\lambda^m+(-\lambda)^{-m})^2-4(-1)^{m}$. 
\end{proof}

\begin{remark} Since the determinant of $M_1$ is $-1$, $D=(\lambda^m+(-\lambda)^{-m})^2-4(-1)^{m}=(\lambda^m-(-\lambda)^{-m})^2$ is not a square of an integer. So, Pell's equation $x^2-Dy^2=\pm 4$ has infinitely many solutions. Moreover, note that the Lucas sequence $\{V_n(\lambda^m+(-\lambda)^{-m}), (-1)^{m})\}=\{V_n(\mu^n+(-\mu)^{-n}, (-1)^{n})\}$, where $\mu, -1/\mu$ are the eigenvalues of $M_2$. 

This result is also known by Alekseyev in \cite[Theorem 8]{A}. He found the solutions of \eqref{the first system} by finding $D=(P_1^2+4)(P_2^2+4)/(\gcd(P_1^2+4, P_2^2+4))$ of Pell's equation.
\end{remark}

\begin{theorem}\label{intersection of V_n2}
Consider the system of equations,
\begin{equation}\label{the second system}
  \left\{\begin{array}{@{}l@{}}
    x^2-(P_1^2-4)y^2= 4\\
    x^2-(P_2^2-4)z^2= 4
  \end{array}\right.\,, (P_1\neq P_2 ~\ \text{ and } ~\ P_1, P_2\geq 4).
\end{equation}
\begin{enumerate}
\item If $(P_1^2-4)(P_2^2-4)$ is not square, then the only solution is $(x,y,z)=(2,0,0)$.
\item If $(P_1^2-4)(P_2^2-4)$ is a square then for any solution $(x,y,z)$, there are even numbers $m,n\in \Z_{>0}$ such that $\trace(N_1^{2n})=\trace(N_2^{2m})$, where 
\begin{equation*}
N_1=\begin{pmatrix} 0& 1\\ -1&P_1\end{pmatrix} \text{ and } \hspace{0.2mm} N_2=\begin{pmatrix} 0& 1\\ -1&P_2\end{pmatrix}.
\end{equation*} 

Moreover, if $m$ and $n$ are the minimal such numbers, then the solution set of the $x$-coordinate forms a Lucas sequence $\{V_n(\lambda^m+\lambda^{-m}, 1)\}$, where $\lambda$ and $1/\lambda$ are the eigenvalues of $N_1$.
\end{enumerate}
\end{theorem}

\begin{proof}
Using the same argument of the proof of Theorem \ref{the first system}, we have the result. In this case, the discriminants of the Pell equations in \eqref{the second system} are $P_i^2-4$, $i=1,2$, and hence we use the matrix representation in \eqref{generalized Fibonacci by matrix with Q=1}. 
    
If $(P_1^2-4)(P_2^2-4)$ is not square, we have only one solution $(x,y,z)=(2,0,0)$. 
For the case where $(P_1^2-4)(P_2^2-4)$ is a square, we have $x^2-2=(P_1^2-4)y^2+2=(P_2^2-4)z^2+2$. Now, by Theorem \ref{correspondence of three items}, $x^2-2$ is the trace of $N_1^{2m}=N_2^{2n}$ for some $m$ and $n$. We have only $x^2-2$ forms of traces of powers of $M_i$, hence, the corresponding automorphisms are all symplectic, and $m$ and $n$ are even numbers.

Suppose that $m$ and $n$ are the minimal numbers such that the trace of $N_1^{2m}$ is the same as the trace of $N_2^{2n}$. Then, by the similar argument of proof of Theorem \ref{the first system}, the solution set of the $x$-coordinate forms a Lucas sequence. Indeed, the first solution in the $x$-coordinate is $x_0=2$. If we let $\lambda, \lambda^{-1}$ and $\mu, \mu^{-1}$ be the eigenvalues of $N_1$ and $N_2$, respectively, then the next two solutions $x_1, x_2$ in the $x$-coordinate satisfy 
\begin{align*}
x_1^2-2=&\trace(N_1^{2m})=b_{2m-1}+b_{2m+1}=(P_1^2-4)b_m^2+2=\lambda^{2m}+\lambda^{-2m},\\
x_2^2-2=&\trace(N_1^{4m})=b_{4m-1}+b_{4m+1}=(P_1^2-4)b_{2m}^2+2=\lambda^{4m}+\lambda^{-4m},
\end{align*}
where $b_k=U_k(P_1, 1)$, i.e., $b_0=0, b_1=1, b_k=P_1\cdot b_{k-1}-b_{k-2}$.

Hence 
\begin{align*}
x_0=2,& ~\ x_1=\lambda^{m}+\lambda^{-m}, \\
x_2=\lambda^{2m}+\lambda^{-2m},& ~\ x_k=\lambda^{km}+\lambda^{-km} ~\ \text{ for } k\in \Z_{\geq 0}.
\end{align*}

So, the set of solutions in the $x$-coordinate of \eqref{the second system} is the Lucas sequence $\{V_n(\lambda^m+\lambda^{-m}, 1)\}$. Moreover, it satisfies Pell's equation $x^2-Dy^2=\pm 4$, where $D=(\lambda^m+\lambda^{-m})^2-4$. 
\end{proof}

\begin{remark} Since the determinant of $M_1$ is $1$, $D=(\lambda^m+\lambda^{-m})^2-4=(\lambda^m-\lambda^{-m})^2$ is not a square of an integer. So, Pell's equation $x^2-Dy^2=\pm 4$ has infinitely many solutions. Moreover, note that the Lucas sequence $\{V_n(\lambda^m+\lambda^{-m}, 1)\}=\{V_n(\mu^n+\mu^{-n}, 1)\}$, where $\mu$ and $1/\mu$ are the eigenvalues of $N_2$. 
\end{remark}

\begin{theorem}\label{intersection of V_n3}
Consider a system of Diophantine equations,
\begin{equation}\label{the third system}
  \left\{\begin{array}{@{}l@{}}
    x^2-(a^2+4)y^2=4\\
    x^2-(b^2-4)z^2= 4, ~\ (b\geq 4)
  \end{array}\right.\,.
\end{equation}
\begin{enumerate}
\item If $(a^2+4)(b^2-4)$ is not square, then the only solution is $(x,y,z)=(2,0,0)$.
\item If $(a^2+4)(b^2-4)$ is a square then for any solution $(x,y,z)$,  there are $m,n\in \Z_{>0}$ such that $\trace(M^{2m})=\trace(N^{2n})$, where 
\begin{equation*}
M=\begin{pmatrix} 0& 1\\ 1&a\end{pmatrix} \text{ and } \hspace{0.2mm} N=\begin{pmatrix} 0& 1\\ -1&b\end{pmatrix}.
\end{equation*} Moreover, if $m$ and $n$ are the minimal such numbers, then the solutions of the system in the $x$-coordinate form a Lucas sequence $\{V_n(\lambda^m+\lambda^{-m}, 1)\}$, where $\lambda, -1/\lambda$ are the eigenvalues of $M$.
\end{enumerate}
\end{theorem}

\begin{proof}
Using the same argument of the proof of Theorem \ref{the first system}, we have the result. 
    
If $(a^2+4)(b^2-4)$ is not square, we have only one solution $(x,y,z)=(2,0,0)$. 
For the case where $(a^2+4)(b^2-4)$ is square, we have $x^2-2=(a^2+4)y^2+2=(b^2-4)z^2+2$. Now, by Theorem \ref{correspondence of three items}, $x^2-2$ is the trace of $M^{2m}=N^{2n}$ for some $m$ and $n$. We have only $x^2-2$ forms of traces of powers of $M$, hence, the corresponding automorphisms are all symplectic, and $m$ is an even number.

Suppose that $m$ and $n$ are the minimal numbers such that $\trace M^{2m}=\trace N^{2n}$. Then, by the similar argument of proof of Theorem \ref{the first system}, the solution set of the $x$-coordinate forms a Lucas sequence. Indeed, the first solution in the $x$-coordinate is $x_0=2$. If we let $\lambda, -\lambda^{-1}$ and $\mu, \mu^{-1}$ be the eigenvalues of $M$ and $N$, respectively, then the next two solutions $x_1, x_2$ in the $x$-coordinate satisfy 
\begin{align*}
x_1^2-2=&\trace(M^{2m})=a_{2m-1}+a_{2m+1}=(a^2+4)a_m^2+2=\lambda^{2m}+\lambda^{-2m},\\
x_2^2-2=&\trace(M^{4m})=a_{4m-1}+a_{4m+1}=(a_1^2+4)a_{2m}^2+2=\lambda^{4m}+\lambda^{-4m},
\end{align*}
where $a_k=U_k(a, -1)$, i.e., $a_0=0, a_1=1, a_k=a\cdot a_{k-1}+a_{k-2}$.

Hence 
\begin{align*}
x_0=2,& ~\ x_1=\lambda^{m}+\lambda^{-m}, \\
x_2=\lambda^{2m}+\lambda^{-2m},& ~\ x_k=\lambda^{km}+\lambda^{-km} ~\ \text{ for } k\in \Z_{\geq 0}.
\end{align*}

So, the solution set in the $x$-coordinate of \eqref{the third system} is the Lucas sequence $\{V_n(\lambda^m+\lambda^{-m}, 1)\}$. Moreover it satisfies the Pell's equation $x^2-Dy^2=\pm 4$, where $D=(\lambda^m+\lambda^{-m})^2-4$. 
\end{proof}

\begin{remark} Since the determinant of $M$ is $-1$, $D=(\lambda^m+\lambda^{-m})^2-4=(\lambda^m-\lambda^{-m})^2$ is not a square of an integer. So Pell's equation $x^2-Dy^2=4$ has infinitely many solutions. Moreover, note that $\{V_n(\lambda^m+\lambda^{-m}, 1)\}=\{V_n(\mu^n+\mu^{-n}, 1)\}$, where $\mu$ and $1/\mu$ are the eigenvalues of $N$. 
\end{remark}

\begin{remark}
The following system of equations
\begin{equation}\label{the fourth system}
  \left\{\begin{array}{@{}l@{}}
    x^2-(P_1^2+4)y^2=\pm 4\\
    x^2-(P_2^2+4)z^2=\mp 4
  \end{array}\right.\,, (P_1\neq P_2),
\end{equation}
has only finite solutions by Theorem \ref{system of Diophantine equations}.
\end{remark}

\subsection*{Acknowledgements}
The author was supported by the National Research Foundation of Korea(NRF-2022R1A2B5B03002625) grant funded by the Korea government(MSIT).




\end{document}